\documentclass[11pt,reqno,oneside]{amsart}
\usepackage{amsthm,amsmath,amssymb,amscd,thmtools}
\usepackage{fullpage}
\usepackage{graphicx} 
\usepackage{subcaption}

\usepackage[a4paper, total={6in, 9.1in}]{geometry}
\usepackage{hyperref}
\hypersetup{
  colorlinks,
  linkcolor={blue!80!black},
  urlcolor={blue!80!black},
  citecolor={blue!80!black}
}
\usepackage[capitalize,noabbrev]{cleveref}
\usepackage{amsfonts, mathrsfs}
\usepackage{xcolor}
\usepackage[normalem]{ulem}

\usepackage[utf8]{inputenc}
\usepackage{latexsym}
\usepackage[all]{xy}
\usepackage{mathtools}

\newtheorem{theorem}{Theorem}[section]

\newtheorem{lemma}[theorem]{Lemma}

\theoremstyle{definition}
\newtheorem{remark}[theorem]{Remark}

\numberwithin{equation}{section}

\title{Moving Anchor  extragradient methods for smooth structured minimax problems}
\author{James K. Alcala}
\thanks{James K. Alcala was partially supported by a UCR Dissertation Year Program Award.}
\address{Department of Mathematics, University of California, Riverside, USA}
\email{jalca014@ucr.edu}

\author{Yat Tin Chow}
\thanks{Professor Yat Tin Chow is partially supported by a Regents' Faculty Fellowship at the University of California, Riverside, and by grants NSF DMS-2409903 and ONR N00014-24-S-B001.}
\address{Department of Mathematics, University of California, Riverside, USA}
\email{yattinc@ucr.edu}

\author{Mahesh Sunkula}
\address{Department of Mathematics, Purdue University, West Lafayette, Indiana, USA}
\email{msunkula@purdue.edu}

\begin{document}

\begin{abstract}
    This work introduces a moving anchor acceleration technique to extragradient algorithms for smooth structured minimax problems. The moving anchor is introduced as a generalization of the original algorithmic anchoring framework, i.e. the EAG method introduced in \cite{yoon_ryu}, in hope of further acceleration. We show that the optimal order of convergence in terms of worst-case complexity on the squared gradient, $O(1/k^{2}),$ is achieved by our new method (where $k$ is the number of iterations). We have also extended our algorithm to a more general nonconvex-nonconcave class of saddle point problems using the framework of \cite{lee2021fast}, which slightly generalizes \cite{yoon_ryu}. We obtain similar order-optimal complexity results in this extended case. In both problem settings, numerical results illustrate the efficacy of our moving anchor algorithm variants, in particular by attaining the theoretical optimal convergence rate for first order methods, as well as suggesting a better optimized constant in the big O notation which surpasses the traditional fixed anchor methods in many cases. A proximal-point preconditioned version of our algorithms is also introduced and analyzed to match optimal theoretical convergence rates.
\end{abstract}
\maketitle
\section{Introduction}

Minimax, min-max, or saddle point problems of the form
\begin{align}
    \min_{x\in \mathbb{R}^{n}}\max_{y\in \mathbb{R}^{m}} L(x,y) \label{problem_setting}
\end{align}
have received considerable attention from optimization researchers and, in particular, machine learning practitioners because of applications including but not limited to Game Theory, Online Learning, GANs \cite{goodfellow2014generative}, \cite{chavdarova2019reducing}, adversarial learning \cite{madry2017towards}, and reinforcement learning \cite{du2017stochastic}.
Measuring the duality gap $\sup_{y^{*}* \in \mathbb{R}^{m}} L(x,y^{*}*) - \inf_{x^{*}* \in \mathbb{R}^{n}} L(x^{*}*,y)$ on averaged (ergodic) iterates or last-iterates of algorithms is one natural way to measure the suboptimality of methods designed to solve \eqref{problem_setting}. 
This is a clear analog to measuring suboptimality for algorithms for minimization problems. 
On the other hand, such a measurement is not as natural to consider when \eqref{problem_setting} is nonconvex-nonconcave, and as will be discussed, the convergence guarantees for this kind of measure may be limiting.

When problem \eqref{problem_setting} is differentiable, another meaningful measure of suboptimality is the squared gradient norm or Hamiltonian of $L$, $\text{Ham}_{L}(x,y) = \|\nabla L(x,y)\|^{2}.$ 
(Sometimes this includes an extra factor of $\frac{1}{2}$, which is not included in this paper. No physical interpretation of this quantity is used here.)
This suboptimality measure retains meaning for nonconvex-nonconcave problems and convergence rates on the squared gradient-norm have only recently attained order-optimal convergence rates in these problem settings.
This is especially important, as many machine learning settings involve neural networks which result in problems that are inherently nonconvex-nonconcave - and as our results indicate, there may still be room for numerical improvements.

The EAG (extra-anchored gradient) class of algorithms, first introduced in \cite{yoon_ryu}, combinesd extragradient and the more recently developed anchoring methods in a single framework to tackle smooth-structured convex-concave minimax problems.
With the primary assumptions being $R-$smoothness and convexity-concavity of \eqref{problem_setting}, EAG achieved $O(1/k^{2}) = \Omega(1/k^{2})$ convergence rates on the squared gradient-norm; that is, the algorithm is order-optimal.
This achievement has inspired a flurry of research activity in recent years \cite{lee2023accelerating}, \cite{suh2023continuous}, \cite{yoon_ryu}.
To show optimality, the authors of \cite{yoon_ryu} adaopt arguments from \cite{nemirovsky1991optimality}, \cite{nemirovsky1992information} to construct a worse-case analysis for a large class of algorithms that contain EAG.

\subsection{Motivations and Contributions} \; \newline
\indent As anchoring is relatively new compared to extragradient, much of the literature written as a direct consequence of these results emphasizes anchoring and other Halpern adjacent techniques \cite{lee2021semi}, \cite{tran2021halpern}, \cite{tran2022connection}.
However, the EAG class is not without limitations.
The two sub-variants of EAG, EAG-V with varying step-size and EAG-C with constant step-size, have difficult convergence analyses and are both relegated to the convex-concave class of smooth functions.
Addressing some of these issues, the authors of \cite{lee2021fast} introduced the Fast ExtraGradient Method, or FEG.
This method generalizes the results of EAG and EG + \cite{diakonikolas2021efficient} to introduce the order-optimal pairing of the extragradient anchor to the setting of certainthe nonconvex-nonconcave problems (specifically, negative comonotone) and introduces an analysis dependent on terms that are less difficult to work with.
Furthermore, their work improves upon the bounding constant attained in EAG in convex-concave problems while retaining optimal convergence rates for a broader class of problems that are of particular importance to machine learning practitioners, among many others.

In the spirit of these previous works, our contributions are as follows.
\begin{enumerate}
    \item We introduce a new technique, the `called moving anchor,' into the algorithmic settings of EAG-V and FEG under minimal assumptions.
    We demonstrate that in both settings, introducing the moving anchor retains order-optimal $O(1/k^{2})$ convergence rates across a range of parameter choices that using the moving anchor gives one access to. One may recover the original fixed-anchor algorithms via parameter tuning, so our algorithms generalize much of the current anchoring literature.
    \item For both the EAG-V moving anchor and the FEG moving anchor, we run a variety of numerical examples by comparing multiple versions of our moving anchor algorithms with their fixed anchor counterparts.
    These numerical examples demonstrate the efficacy of our algorithm, as in all cases, aone of the moving anchor algorithm variantversions in each example is the fastest algorithm by a constant or is comparable-to-better for all iterations.
    In addition, in many cases the fastest moving anchor algorithm appears to have a massive initial oscillation towards the fixed point that the fixed point algorithms seem to lack - this may be beneficial for reaching certain stopping criteria very quickly.
    \item We develop a theoretical version of the moving anchor algorithms (in both the convex-concave EAG-V and nonconvex-nonconcave FEG) with a proximal anchoring step with fruitful implications for future research.
\end{enumerate}

\section{Prior Work, Literature Review \& Preliminaries}

\subsection{Halpern iteration and anchoring} \; \newline
\indent Introduced in 1967 and inspired by Browder's classical fixed point theorem, the Halpern iteration \cite{halpern1967fixed} is an algorithm built for approximating fixed point(s) of nonexpanding maps in a Hilbert space. Its convergence has been studied in \cite{lieder2021convergence}, and it is extensively used in monotone inclusion-type problem settings \cite{diakonikolas2020halpern}, \cite{tran2021halpern}, \cite{cai2022stochastic}. A recent paper \cite{tran2022connection} draws an explicit connection between Halpern-inspired methods and Nesterov's AGM \cite{nesterov1983method_AGM}, linking two very active strains of acceleration literature.

Directly inspired by Halpern, algorithmic anchoring was recently introduced in the literature \cite{ryu2019ode} and has since been utilized to establish optimal $O(1/k^{2})$ convergence rates for smooth-structured convex-concave minimax problems \cite{yoon_ryu}. 
Since then, these methods have been extended to the nonconvex-nonconcave, negative comonotone problem setting \cite{lee2021fast} and analogous settings for composite problems in a multi-step framework \cite{lee2021semi}. 
Interestingly, this latter framework introduces `semi'-anchoring, where only one part of the descent-ascent step is anchored, and a unique anchor occurs at each step of the multi-step. 
To our knowledge, this is the first instance of an anchoring method that goes beyond a single fixed anchor. 
In \cite{tran2021halpern}, the authors develop an anchored Popov's scheme and a splitting version of the EAG developed in \cite{yoon_ryu}, with a similar analysis.

\subsection{Extragradient methods}
The extragradient method first appeared in \cite{korpelevich1976extragradient} and has since been an important acceleration method extensively studied in the optimization literature \cite{azizian2020tight}, \cite{tseng1995linear}, \cite{liu2019towards}, especially in the context of generative adversarial networks \cite{goodfellow2014generative}, \cite{chavdarova2019reducing} and adversarial training \cite{madry2017towards}.
A classical result regarding these methods is that if $X \in \mathbb{R}^{n}, Y\in \mathbb{R}^{m}$ are compact domains, then for the duality gap $\max_{y* \in Y} L(x,y*) - \min_{x* \in X} L(x*,y),$ the ergodic iterate of extragradient-type methods \cite{nemirovski2004prox}, \cite{nesterov2007dual} have an $O(1/k)$ rate, which is order-optimal \cite{ouyang2021lower}, \cite{nemirovskij1983problem}.
Recently, it was shown that the last iterate convergence rate for extragradient also attains $O(1/k)$ convergence \cite{gorbunov2022extragradient}, with \textit{only} monotonicity and Lipschitz assumptions. 
This closes the gap between the last-iterate and ergodic-iterate convergence rates for extragradient discussed in \cite{golowich}.
Another recent interesting result was attained in \cite{diakonikolas2021efficient}, where the authors developed the Extragradient$+$ method, a variant of extragradient extended to various nonconvex-nonconcave problem settings.

On the other hand, when the problem at hand has certain smoothness properties, the squared gradient norm $\|\nabla L\|^{2}$ for extragradient-type algorithms recently achieved order-optimal convergence of $O(1/k^{2})$ \cite{yoon_ryu}, \cite{lee2021fast}, thanks in part to the synthesis with anchoring.
This breaks the bound of the SCLI class of algorithms discussed in \cite{golowich}, which contains the unmodified extragradient, because EAG is \textit{not} SCLI, but specifically 2-CLI or in an extended class of 1-CLI algorithms.
See Appendix D.2 of \cite{yoon_ryu} for a best-iterate (NOT last iterate, at the time of writing this quantity doesn't seem to be known) convergence analysis of extragradient and Appendix E of \cite{yoon_ryu} and \cite{golowich} for more details on the relationships between these classes of algorithms. 
We conclude this discussion by remarking that for smooth problems, the bound on the squared gradient norm is meaningful in nonconvex-nonconcave problem settings, and as demonstrated in this work and these recent works, has room for numerical improvement.

\subsection{Preliminaries}

A saddle function $L: \mathbb{R}^{n} \times \mathbb{R}^{m} \to \mathbb{R}$ is (non)convex-(non)concave if it is (non)convex in $x$ for any fixed $y\in\mathbb{R}^{m}$ and (non)concave in $y$ for any fixed $x\in\mathbb{R}^{n}$.
A saddle point $(\hat{x}, \hat{y}) \in \mathbb{R}^{n} \times \mathbb{R}^{m}$ is any point such that the inequality $L(\hat{x}, y) \leq L(\hat{x}, \hat{y}) \leq L(x, \hat{y})$ for all $x \in \mathbb{R}^n$ and $y\in \mathbb{R}^{m}.$
Solutions to \eqref{problem_setting} are defined as saddle points.

Throughout this paper, we assume the differentiability of $L$, and we are especially interested in the so-called \textit{saddle operator} associated to $L$,
\begin{align} \label{sad_operator} G_{L}(z) &= \left[
    \begin{array}{c}
    \nabla_{x} L(x,y)\\
    -\nabla{y} L(x,y)
\end{array}
\right] \end{align}
where the $L$ subscript is omitted when the underlying saddle function is known.
When our problem is convex-concave, the operator \eqref{sad_operator} is known to be monotone \cite{rockafellar1970monotone}, meaning $\langle G_{L}(z_{1}) - G_{L}(z_{2}),z_{1} - z_{2} \rangle \geq 0 \;\forall z_{1}, z_{2} \in \mathbb{R}^{n} \times \mathbb{R}^{m}.$
We assume that this operator $G_{L}$ is $R$-Lipschitz, or has certain stronger Lipschitz properties we detail later; this is sometimes referred to as $L$ being $R$-smooth.
With these properties in mind, one may introduce an assumption that generalizes monotonicity: let $\rho \in (-\frac{1}{2R}, +\infty).$ 
In this paper, we assume that when $G_{L}$ is \textit{not} monotone, it satisfies
\[
    \langle G_{L}(z_{1}) - G_{L}(z_{2}),z_{1} - z_{2} \rangle \geq  \rho \|G_{L}(z_{1}) - G_{L}(z_{2})\|^{2} \; \forall z_{1}, z_{2} \in \mathbb{R}^{n} \times \mathbb{R}^{m}.
\]
When $\rho >0,$ this is called co-coercivity; when $\rho=0,$ this recovers monotonicity; when $\rho<0,$ this is called negative comonotonicity.
This latter condition on \eqref{sad_operator} allows one to consider certain nonconvex-nonconcave problems $L$, and is also going to be a central focus of this work.
Note, however, that these assumptions need not cover all smooth nonconvex-nonconcave problems of interest.
Figure 1, Table 1, and Example 1 of \cite{lee2021fast} illustrate broader problem classes than negative comonotonicity that retain smoothness while being nonconvex-nonconcave.

Finally we state that although $\nabla L \ne G_{L},$ we have $\|\nabla L\| = \|G_{L}\|,$ so we may use these expressions interchangeably.

\section{Original Algorithm, EAG-V}

The Extra Anchored Gradient Algorithm, or EAG with varying step size (EAG-V) has a simple statement and a relatively simple proof of convergence:

\begin{align}
    \label{original_eagv1} z^{k+1/2}   &= z^{k} + \beta_{k}(z^{0} - z^{k}) - \alpha_{k}G(z^{k})\\
    \label{original_eagv2} z^{k+1}     &= z^{k} + \beta_{k}(z^{0} - z^{k}) - \alpha_{k}G(z^{k + 1/2})\\
    \notag \alpha_{k+1}&= \frac{\alpha_{k}}{1 - \alpha_{k}^{2}R^{2}}\left(1 - \frac{(k+2)^{2}}{(k+1)(k+3)}\alpha_{k}^{2}R^{2}\right)\\
    \label{original_alpha} &= \alpha_{k}\left(1 - \frac{1}{(k+1)(k+3)}\frac{\alpha_{k}^{2}R^{2}}{1 - \alpha_{k}^{2}R^{2}}\right)
\end{align}

with $\alpha_{0} \in (0,1/R),$ and R a predetermined constant.
Here, $G$ is the so-called saddle operator, $G:=(\nabla_{x}L,-\nabla_{y}L)$ and $L$ is a convex-concave saddle function in a minimax optimization problem.
It is a nontrivial fact that $G$ is monotone \cite{yoon_ryu}.
The structure of the $\alpha_{k}$'s and $\beta_{k}$'s are detailed below alongside auxiliary sequences $A_k$ and $B_k$. 
We state the convergence of this algorithm as a theorem and relay the details of its convergence via a specific Lyapunov functional as a lemma.
For more details, including a version of EAG with a non-varying step size, see \cite{yoon_ryu}.
\begin{theorem}[EAG-V convergence rate \cite{yoon_ryu}]
    Assume $L: \mathbb{R}^{n} \times \mathbb{R}^{m} \to \mathbb{R}$ is an $R$-smooth convex-concave function with a saddle point $z^{*}.$
    Assume further that $\alpha_{0} \in (0,\frac{3}{4R})$ and define $\alpha_{\infty} = \lim_{k\to\infty}\alpha_{k}.$
    Then EAG-V converges, with rate
    \[
    ||\nabla G(z^{k})||^{2} \leq \frac{4(1 + \alpha_{0}\alpha_{\infty}R^{2})}{\alpha_{\infty}^{2}}\frac{||z^{0} - z^{*}||^{2}}{(k+1)(k+2)}
    \]
    where $G = (\nabla L|_{x \in \mathbb{R}^{n}}, -\nabla L|_{-y \in -\mathbb{R}^{m}}).$
\end{theorem}

Since $z^{*}$ is the saddle point, this theorem demonstrates $O(1/k^{2})$ convergence of the algorithm. To derive this order of convergence, the following lemma is necessary.

\begin{lemma}[EAG Lyapunov Functional \cite{yoon_ryu}] \label{fixed_anchor_Lyapunov_lemma}
    Let $\{\beta_{k}\}_{k\geq0} \subseteq (0,1)$ and $\alpha_{0} \in (0,\frac{1}{R})$ be given.
    Consider the following sequences defined by the given recurrence relations for $k \geq 0:$
    \begin{align}
        \label{first_A} A_{k} &= \frac{\alpha_{k}}{2\beta_{k}}B_{k}\\
        \label{first_B} B_{k+1} &= \frac{B_{k}}{1 - \beta_{k}}\\
        \label{first_alpha} \alpha_{k+1} &= \frac{ \alpha_{k} \beta_{k+1}(1 - \alpha_{k}^{2}R^{2} - \beta_k^{2}) }{ \beta_{k}(1 - \beta_{k})(1 - \alpha_{k}^{2}R^{2}) }
    \end{align}
    where $B_{0} = 1.$
    Assume that $\alpha_{k} \in (0,\frac{1}{R})$ holds for all $k \geq 0,$ and that $L$ is $R-$smooth and convex-concave.
    Then the sequence $\{V_{k}\}_{k\geq 0}$ defined as
    \begin{align}
        \label{original_fxnal} V_{k} &:=A_{k}\|G(z^{k})\|^{2} + B_{k}\langle G(z^{k}),z^{k} - z^{0}\rangle
    \end{align}
    is non-increasing.
\end{lemma}

Within \eqref{original_fxnal}, choosing $\beta_{k} = \frac{1}{k+2}$ yields $B_{k} = k+1, A_{k} = \frac{\alpha_{k}(k+2)(k+1)}{2},$ and the construction of $\alpha_{k+1}$ in \eqref{first_alpha}.

\section{EAG-V with moving anchor}

In this section, we construct and analyze a new version of the EAG-V algorithm. Here, the anchoring point moves at each time step. We call this the moving anchor algorithm; it utilizes a similar extragradient step. Further down, we demonstrate comparable rates of convergence to the original EAG algorithm with varying step-size.

For the $k-th$ iterate of $z^{0} \in \mathbb{R}^{n} \times \mathbb{R}^{m},$ the EAG-V with moving anchor is defined as 
\begin{align}
    \notag z^{0}       &= \bar{z}^{0}\\
    \label{eagv update1} z^{k+1/2}   &= z^{k} + \frac{1}{k+2}(\bar{z}^{k} - z^{k}) - \alpha_{k}G(z^{k})\\
    \label{eagv update2} z^{k+1}     &= z^{k} + \frac{1}{k+2}(\bar{z}^{k} - z^{k}) - \alpha_{k}G(z^{k + 1/2})\\
    \label{eagv update3} \bar{z}^{k+1}&= \bar{z}^{k} + \gamma_{k+1}G(z^{k+1})
\end{align}

The major structural difference here is the introduction of the regularly-updating $\bar{z}^{k},$ analogous to the role of $z^{0}$ in the EAG-V detailed in the previous section.
\eqref{eagv update3} is the regular update for this anchor; it depends on the algorithm update \eqref{eagv update2} rather than exclusively on itself.
All previously defined terms are the same as in the fixed anchor algorithm, now with

\begin{align}
    \label{c_update_1} c_{k+1}     &\leq \frac{c_{k}}{1 + \delta_{k}},\\
    \label{gamma_update_1} \gamma_{k+1}&\leq \frac{B_{k+1}}{c_{k+1}(1 + \frac{1}{\delta_{k}})}.
\end{align}

We choose $\delta_{k}$ so that $\displaystyle \sum_{k=0}^{\infty} \log(1 + \delta_{k}) < \infty$. 
The $c_{k}$ terms are part of the definition of the Lyapunov functional we use in our analysis; these come in handy when we use $\gamma_{k}$ to absolve terms.
Let $c_{\infty}:= \lim_{k \to \infty} c_{k} = c_{0} \prod_{k=0}^{\infty} \frac{1}{1+\delta_{k}}$. 
As a general rule, one wishes to choose $c_{0}$ so that $c_{\infty}$ satisfies some specified convergence constraint; these constraints will appear throughout the major convergence theorems in this section and the next section.
While the choice of $c_{0}$ is therefore limited to according to certain problem/algorithm constraints, in general there seems to be much freedom in choosing $c_{0}$ and the sequence $\{\delta_{k}\}.$ 
For the rest of this article, we take \eqref{c_update_1} and \eqref{gamma_update_1} to be given with equal signs instead of inequalities and as in the fixed anchor case, we will take $\beta_{k} = \frac{1}{k+2},$ resulting in similar sequences \eqref{first_A}, \eqref{first_B}, \eqref{first_alpha} for the moving anchor.
Before we proceed with the analysis, we emphasize that the original EAG-V algorithm may be recovered simply by setting $\gamma_{k+1} := 0$ for all $k.$

Now, we give the definition of the Lyapunov functional and show that it is nonincreasing:

\begin{lemma}\label{LpnvLm_1}
    The Lyapunov functional
    \begin{equation}
        \label{first_functional} V_{k} := A_{k}\|G(z^{k})\|^{2} + B_{k}\langle G(z^{k}),z^{k} - \bar{z}^{k}\rangle +  c_{k}\|z^{*} - \bar{z}^{k}\|^{2},
    \end{equation}
    corresponding to the moving anchor EAG-V algorithm \eqref{eagv update1} through \eqref{eagv update3} with constants $A_k, B_k, c_k, \beta_{k}$ defined as in \eqref{first_A}, \eqref{first_B}, \eqref{first_alpha}, and \cref{fixed_anchor_Lyapunov_lemma}, along with sequences $c_{k}, \gamma_{k}$ defined in \eqref{c_update_1}, \eqref{gamma_update_1}, is non increasing.
\end{lemma}
\begin{proof}

First we reorganize some of the algorithm statements and label them for use later.
\begin{align}
    \quad& z^{k} - z^{k+1} = \beta_{k}(z^{k} - \bar{z}^{k}) + \alpha_{k}G(z^{k+1/2}) \label{eagv re update1}\\
    \label{eagv re update2} \quad& z^{k+1/2} - z^{k+1} = \alpha_{k}(G(z^{k+1/2}) - G(z^{k}))\\
    \label{eagv re update3} \quad& \bar{z}^{k} - z^{k+1} = (1 - \beta_{k})(\bar{z}^{k} - z^{k}) + \alpha_{k}G(z^{k+1/2})\\
    \label{eagv re update4} \quad& \bar{z}^{k} - \bar{z}^{k+1} = -\gamma_{k+1}G(z^{k+1})
\end{align}
\eqref{eagv re update1} comes from rearranging \eqref{eagv update2}, \eqref{eagv re update2} comes from taking the difference between \eqref{eagv update1} and \eqref{eagv update2}, \eqref{eagv re update3} is $\bar{z}^{k}$ minus \eqref{eagv update2}, and \eqref{eagv re update4} is \eqref{eagv update3} rearranged. 
The overall goal of this proof is to show that the difference $V_{k} - V_{k+1}$ is nonnegative.
\begin{align*}
    V_{k} &- V_{k+1}\\
    \geq &A_{k}\|G(z^{k})\|^{2} - A_{k+1}\|G(z^{k+1})\|^{2} \underbrace{+B_{k}\langle z^{k} - \bar{z}^{k},G(z^{k})\rangle}_{\text{I}} \\
    &\underbrace{- B_{k+1}\langle z^{k+1} - \bar{z}^{k+1},G(z^{k+1})\rangle}_{\text{II}}
    +c_{k}\|z^{*} - \bar{z}^{k}\|^{2} - c_{k+1}\|z^{*} - \bar{z}^{k+1}\|^{2}\\
    & \underbrace{-\frac{B_{k}}{\beta_{k}}\langle z^{k} - z^{k+1},G(z^{k}) - G(z^{k+1})\rangle}_{\text{III}}
\end{align*}
    Notice that the last term above, III, is not part of the definition of $V_{k}$ nor $V_{k+1}.$ It has been introduced to aid in the proof and is nonnegative by the monotonicity of $G$. We would like to absolve any terms containing the $\bar{z}^{k}, \bar{z}^{k+1}$ terms. To accomplish this, our next goal is to focus on turning the labeled parts (I, II, III) of the above line into
\[
    \underbrace{\alpha_{k}B_{k+1}\langle G(z^{k + 1/2},G(z^{k + 1})\rangle + \frac{B_{k+1}}{\gamma_{k+1}}\|\bar{z}^{k} - \bar{z}^{k + 1}\|^{2} - \frac{\alpha_{k}B_{k}}{\beta{k}}\langle G(z^{k + 1/2}),G(z^{k})-G(z^{k + 1})\rangle}_{\text{IV}}.
\]
We now detail this process.
The term I does not change. For II, on the other hand, we have
\begin{align}
    &\underbrace{-B_{k+1}\langle z^{k+1}- \bar{z}^{k+1},G(z^{k+1})\rangle}_{\text{II}}\notag\\
    \label{for_neg_gamma1} &=B_{k+1}\langle \bar{z}^{k} - z^{k+1},G(z^{k+1})\rangle - B_{k+1}\langle \bar{z}^{k} - \bar{z}^{k+1},G(z^{k+1})\rangle\\
    \label{for_neg_gamma2} &=B_{k+1}\langle (1-\beta_{k})(\bar{z}^{k} - z^{k}) + \alpha_{k}G(z^{k+1/2}),G(z^{k+1})\rangle - B_{k+1}\langle -\gamma_{k+1}G(z^{k+1}),G(z^{k+1})\rangle
\end{align}
where the first equality comes from recognizing $z^{k+1} - \bar{z}^{k+1} = z^{k+1} - \bar{z}^{k} + \bar{z}^{k} - \bar{z}^{k+1}$ and the second comes from substituting in equality \eqref{eagv re update3} and \eqref{eagv re update4}. For III, 
\begin{align}
    \notag &\underbrace{-\frac{B_{k}}{\beta_{k}}\langle z^{k} - z^{k+1},G(z^{k}) - G(z^{k+1})\rangle}_{\text{III}}\\
    \label{middle_brown_term_line}&=-\frac{B_{k}}{\beta_{k}}\langle z^{k} - z^{k+1},G(z^{k})\rangle +\frac{B_{k}}{\beta_{k}}\langle z^{k} - z^{k+1},G(z^{k+1})\rangle\\
    \notag &=-\frac{B_{k}}{\beta_{k}}\langle\beta_{k}(z^{k} - \bar{z}^{k}) + \alpha_{k}G(z^{k+1/2}),G(z^{k})\rangle +\frac{B_{k}}{\beta_{k}}\langle\beta_{k}(z^{k} - \bar{z}^{k}) + \alpha_{k}G(z^{k+1/2}),G(z^{k+1})\rangle,
\end{align}
where the last equality is a result of substituting in \eqref{eagv re update1} in each of the first arguments of the two terms in \eqref{middle_brown_term_line}.
Now, we can begin simplify everything we've done to obtain IV.
\begin{align}
    \label{formerly_a}&\underbrace{\langle z^{k} - \bar{z}^{k},G(z^{k})\rangle}_{\text{I}}\\
    \label{formerly_b}&\underbrace{\langle(1 - \beta_{k})(z^{k} - \bar{z}^{k}) -\alpha_{k}G(z^{k+1/2})-\gamma_{k+1}G(z^{k+1}),G(z^{k+1})\rangle}_{\text{II}}\\
    \label{formerly_c}&\underbrace{-\frac{B_{k}}{\beta_{k}}\langle\beta_{k}(z^{k} - \bar{z}^{k}) + \alpha_{k}G(z^{k+1/2}),G(z^{k})\rangle}_{\text{III}}\\ \label{formerly_d}&\underbrace{+\frac{B_{k}}{\beta_{k}}\langle\beta_{k}(z^{k} - \bar{z}^{k}) + \alpha_{k}G(z^{k+1/2}),G(z^{k+1})\rangle}_{\text{III}}
\end{align}
From here, we'll use two facts. 
First, $B_{k+1} = \frac{B_{k}}{1 - \beta_{k}}.$
This allows us to combine and cancel the very first component of \eqref{formerly_b} with the $\beta_{k}(z^{k} - \bar{z}^{k})$ component of \eqref{formerly_d}.
Additionally, \eqref{formerly_a} cancels with the $\beta_{k}(z^{k} - \bar{z}^{k})$ component of \eqref{formerly_c}. This leaves us with
\begin{align}
    \notag =& \underbrace{\alpha_{k}B_{k+1}\langle G(z^{k+1/2}),G(z^{k+1})\rangle + B_{k+1}\langle\gamma_{k+1}G(z^{k+1}),G(z^{k+1})\rangle}_{\text{II}}\\
    \notag &\underbrace{-\frac{B_{k}\alpha_{k}}{\beta_{k}}\langle G(z^{k+1/2}),G(z^{k})\rangle + \frac{B_{k}\alpha_{k}}{\beta_{k}}\langle G(z^{k+1/2}),G(z^{k+1})\rangle}_{\text{III}}\\
    \notag =&\underbrace{\alpha_{k}B_{k+1}\langle G(z^{k+1/2}),G(z^{k+1})\rangle + \frac{B_{k+1}}{\gamma_{k+1}}\|\bar{z}^{k} - \bar{z}^{k+1}\|^{2}}_{\text{IV}}\\
    \notag &\underbrace{-\frac{\alpha_{k}B_{k}}{\beta_{k}}\langle G(z^{k+1/2}),G(z^{k}) - G(z^{k+1})\rangle}_{\text{IV}},
\end{align}
where the last equality is a result of applying the anchor update to get the norm squared term, and combining the latter two terms while leaving $G(z^{k+1/2})$ fixed.

Thus, we've shown
\begin{align}
    \notag &A_{k}\|G(z^{k})\|^{2} - A_{k+1}\|G(z^{k+1})\|^{2}\\
    \notag +&B_{k}\langle z^{k} - \bar{z}^{k},G(z^{k})\rangle - B_{k+1}\langle z^{k+1} - \bar{z}^{k+1},G(z^{k+1})\rangle\\
    \notag +& c_{k}\|z^{*} - \bar{z}^{k}\|^{2} - c_{k+1}\|z^{*} - \bar{z}^{k+1}\|^{2} - \frac{B_{k}}{\beta_{k}}\langle z^{k} - z^{k+1},G(z^{k}) - G(z^{k+1})\rangle \\
    \label{viable_line_1}=&A_{k}\|G(z^{k})\|^{2} - A_{k+1}\|G(z^{k+1})\|^{2} + \alpha_{k}B_{k+1}\langle G(z^{k+1/2}),G(z^{k+1})\rangle \\
    \label{viable_line_2}-&\frac{\alpha_{k}B_{k}}{\beta_{k}}\langle G(z^{k+1/2}),G(z^{k}) - G(z^{k+1})\rangle \\
    \label{viable_line_3}+&c_{k}\|z^{*} - \bar{z}^{k}\|^{2} - c_{k+1}\|z^{*} - \bar{z}^{k+1}\|^{2} + \frac{B_{k+1}}{\gamma_{k+1}}\|\bar{z}^{k} - \bar{z}^{k+1}\|^{2}
\end{align}
Now, we continue on with our goal of absolving terms.
From Cauchy, we have that
\begin{align}
    \label{cauchy_conv_conc} \quad ||z^{*} - \bar{z}^{k+1}\|^{2} \leq (1 + \delta_{k})\|z^{*} - \bar{z}^{k}\|^{2} + (1+\frac{1}{\delta_{k}})\|\bar{z}^{k} - \bar{z}^{k+1}\|^{2}
\end{align}
and from the algorithm definition,
\begin{align}
    \label{definition_cterm_gammaterm} \quad c_{k+1} = \frac{c_{k}}{1 + \delta_{k}}, \quad \gamma_{k+1} = \frac{B_{k+1}}{c_{k+1}(1 + \frac{1}{\delta_{k}})}.
\end{align}
We apply \eqref{cauchy_conv_conc} to \eqref{viable_line_3} to obtain
\begin{align*}
    \geq&A_{k}\|G(z^{k})\|^{2} - A_{k+1}\|G(z^{k+1})\|^{2} + \alpha_{k}B_{k+1}\langle G(z^{k+1/2}),G(z^{k+1})\rangle\\
    -&\frac{\alpha_{k}B_{k}}{\beta_{k}}\langle G(z^{k+1/2}),G(z^{k}) - G(z^{k+1})\rangle + c_{k}\|z^{*} - \bar{z}^{k}\|^{2}\\
    -&c_{k+1}\big{(}(1 + \delta_{k})\|z^{*} - \bar{z}^{k}\|^{2} + (1+\frac{1}{\delta_{k}})\|\bar{z}^{k} - \bar{z}^{k+1}\|^{2}\big{)} + \frac{B_{k+1}}{\gamma_{k+1}}\|\bar{z}^{k} - \bar{z}^{k+1}\|^{2}
\end{align*}
and now we apply \eqref{definition_cterm_gammaterm}:
\begin{align*}
    \geq&A_{k}\|G(z^{k})\|^{2} - A_{k+1}\|G(z^{k+1})\|^{2} + \alpha_{k}B_{k+1}\langle G(z^{k+1/2}),G(z^{k+1})\rangle\\
    -&\frac{\alpha_{k}B_{k}}{\beta_{k}}\langle G(z^{k+1/2}),G(z^{k}) - G(z^{k+1})\rangle + c_{k}\|z^{*} - \bar{z}^{k}\|^{2}\\
    -&c_{k}\|z^{*} - \bar{z}^{k}\|^{2} - \frac{B_{k+1}}{\gamma_{k+1}}\|\bar{z}^{k} - \bar{z}^{k+1}\|^{2} + \frac{B_{k+1}}{\gamma_{k+1}}\|\bar{z}^{k} - \bar{z}^{k+1}\|^{2}\\
    =&A_{k}\|G(z^{k})\|^{2} - A_{k+1}\|G(z^{k+1})\|^{2} + \alpha_{k}B_{k+1}\langle G(z^{k+1/2}),G(z^{k+1})\rangle\\
    -&\frac{\alpha_{k}B_{k}}{\beta_{k}}\langle G(z^{k+1/2}),G(z^{k}) - G(z^{k+1})\rangle + 0.
\end{align*}

At this point, showing that the remaining terms are nonnegative is nontrivial, but directly follows the arguments made in the proof of Lemma 2 in \cite{yoon_ryu}.
Specifically, following $(29)$ onwards in \cite{yoon_ryu}, one will find that
\begin{align*}
    A_{k}&\|G(z^{k})\|^{2} - A_{k+1}\|G(z^{k+1})\|^{2} + \alpha_{k}B_{k+1}\langle G(z^{k+1/2}),G(z^{k+1})\rangle\\
    -&\frac{\alpha_{k}B_{k}}{\beta_{k}}\langle G(z^{k+1/2}),G(z^{k}) - G(z^{k+1})\rangle\\
    \geq 0,
\end{align*}
which completes the proof.
\end{proof}

Now we have the primary result of this section.
\begin{theorem} \label{conv_conc_convergence_thm}
    The EAG-V algorithm with moving anchor \eqref{eagv update1}, \eqref{eagv update2}, and \eqref{eagv update3} together with the Lyapunov functional $V_{k}$ \eqref{first_functional} described in \cref{LpnvLm_1}, has convergence rate
    \begin{equation}
        \label{convex_concave_bound} \|G(z^{k})\|^{2}\leq \frac{4(\alpha_{0}R^{2} + c_{0})\|z^{0}-z^{*}\|^{2}2V_{0}}{\frac{\alpha_{\infty}}{4}(k+1)(k+2)}
    \end{equation}
    as long as we assume $c_{\infty}\alpha_{\infty} \geq 1.$
\end{theorem}
\begin{proof}
    For the most part, this argument parallels the analogous argument found in \cite{yoon_ryu}.
    We use the Lyapunov functional to isolate and bound $\|G(z^{k})\|^{2}.$
    \begin{align}
        \label{bound_on_Vk} V_{k} &\leq V_{0} = \alpha_{0}\|G(z^{0})\|^{2} + c_{0}\|z^{0}-z^{*}\|^{2}\\
        \notag &\leq (\alpha_{0}R^{2} + c_{0})\|z^{0}-z^{*}\|^{2}
    \end{align}
    by $R-$smoothness.
    On the other hand,
    \begin{align}
        \notag V_{k} &= A_{k}\|G(z^{k})\|^{2} + B_{k}\langle G(z^{k}),z^{k}-\bar{z}^{k}\rangle + c_{k}\|z^{*}-\bar{z}^{k}\|^{2}\\
        \notag &\geq A_{k}\|G(z^{k})\|^{2} + B_{k}\langle G(z^{k}),z^{*}-\bar{z}^{k}\rangle + c_{k}\|z^{*}-\bar{z}^{k}\|^{2}\\
        \notag &\geq \frac{A_{k}}{2}\|G(z^{k})\|^{2} + (c_{k} - \frac{B_{k}^{2}}{2A_{k}})\|z^{*}-\bar{z}^{k}\|^{2}\\
        \notag &=\frac{\alpha_{k}(k+1)(k+2)}{4}\|G(z^{k})\|^{2} + (c_{k} - \frac{k+1}{\alpha_{k}(k+2)})\|z^{*}-\bar{z}^{k}\|^{2}\\
        \notag &\geq \frac{\alpha_{\infty}}{4}(k+1)(k+2)\|G(z^{k})\|^{2} + (c_{\infty} - \frac{1}{\alpha_{\infty}})\|z^{*}-\bar{z}^{k}\|^{2}\\
        \notag &\geq \frac{\alpha_{\infty}}{4}(k+1)(k+2)\|G(z^{k})\|^{2}
    \end{align}
    As long as $c_{\infty} \geq \frac{1}{\alpha_{\infty}},$ the second to last line above is positive, and we may focus on the inequality given to us by the last line above:
    \[
        \frac{\alpha_{\infty}}{4}(k+1)(k+2)\|G(z^{k})\|^{2} \leq (\alpha_{0}R^{2} + c_{0})\|z^{0}-z^{*}\|^{2}.
    \]
    Dividing both sides by the constant $\frac{\alpha_{\infty}}{4}(k+1)(k+2)$ gives the desired result.
\end{proof}

\subsection{Proof of convergence for $-\gamma_{k}$}

We next show that, for a slightly restricted choice of $\gamma_{k},$ our proof works for $-\gamma_{k}$ in place of $\gamma_{k}.$ This is of interest as numerical results indicate that certain problem settings favor $-\gamma_{k}$ in terms of convergence speed by a constant, while $+\gamma_{k}$ seems to be favored in other settings$-\gamma_{k}$ exhibits faster convergence to the saddle point by a constant.

\begin{lemma} \label{neg_gam_lemma1}
    Replacing $\gamma_{k}$ with $-\gamma_{k}$ in the definition of the EAG-V algorithm with moving anchor \eqref{eagv update1}, \eqref{eagv update2}, \eqref{eagv update3}, and suppose $\displaystyle \gamma_{k+1} = \min{\frac{B_{k+1}}{c_{k+1}(1+ \frac{1}{\delta_{k}})},\frac{e_{k+1}}{2B_{k+1}\|G(z^{k+1})\|^{2}}},$ where $\sum e_{k} < \infty.$ Then the Lyapunov functional \eqref{first_functional} is nonincreasing, and has the same order of convergence $O(1/k^{2})$ as in the positive $\gamma_{k}$ moving anchor EAG-V algorithm \eqref{convex_concave_bound}.
\end{lemma}
\begin{proof}
    First, note that the anchor update \eqref{eagv update3} has been modified to become
    \begin{align}
        \label{negative_gamma_update} -\gamma_{k+1} = -\frac{B_{k+1}}{c_{k+1}(1 + \frac{1}{\delta_{k}})},
    \end{align}
    resulting in the following modification to \eqref{eagv re update4}:
    \begin{align}
        \label{modified_neg_update} \bar{z}^{k} - \bar{z}^{k+1} = \gamma_{k+1}G(z^{k+1}).
    \end{align}
    We see the first adjustment in the previous lemma in the transition from line \eqref{for_neg_gamma1} to \eqref{for_neg_gamma2}; note that we focus only on the terms dependent on \eqref{modified_neg_update}:
    \begin{align}
        \notag -&B_{k+1}\langle \bar{z}^{k} - \bar{z}^{k+1}, G(z^{k+1}) \rangle\\
        \notag =-&B_{k+1}\langle \gamma_{k+1}G(z^{k+1}), G(z^{k+1})\rangle\\
        \notag =-&B_{k+1}\langle (2\gamma_{k+1} - \gamma_{k+1})G(z^{k+1}), G(z^{k+1})\rangle\\
        \label{negative_gamma_proof_double_line} =-&B_{k+1}\langle 2\gamma_{k+1}G(z^{k+1}), G(z^{k+1})\rangle + B_{k+1}\langle \gamma_{k+1}G(z^{k+1}),G(z^{k+1})\rangle.
    \end{align}
    The latter term in line \eqref{negative_gamma_proof_double_line} will go on and cancel in a quadratic form as in the proof of the original lemma. Continuing, one will be left over with the term $-B_{k+1}\langle 2\gamma_{k+1}G(z^{k+1}), G(z^{k+1})\rangle$.
    
    At this point, if we proceed as in \cref{LpnvLm_1}, we end up with the inequality 
    \[
        V_{k} - V_{k+1} \geq -2\gamma_{k+1}B_{k+1}\|G(z^{k+1})\|^{2}
    \]
    or, after rearranging,
    \[
        V_{k} - V_{k+1} + 2\gamma_{k+1}B_{k+1}\|G(z^{k+1})\|^{2} \geq 0.
    \]
    By construction, the left-hand side of the inequality should remain nonnegative.
    Now, because
    \[
        \gamma_{k+1} \leq \frac{e_{k+1}}{2B_{k+1}\|G(z^{k+1})\|^{2}},
    \]
    when we proceed as in the proof of \cref{conv_conc_convergence_thm} to show convergence, getting to the line \eqref{bound_on_Vk}, we get the inequality
    \begin{align*}
        V_{k} \leq& V_{0} + \sum_{j = 1}^{k-1} 2\gamma_{j} B_{j} \|G(z^{j})\|^{2}\\
        \leq& V_{0} + \sum_{j = 1}^{k-1} e_{j}\\
        \leq& V_{0} + \sum_{j = 1}^{\infty} e_{j}\\
        =& C V_{0},
    \end{align*}
    where $C$ is a constant. This completes the proof that our algorithm has both a nonincreasing Lyapunov functional and the $O(1/k^{2})$ convergence under the assumption of a (slightly restricted) negative $\gamma_{k}$ term.
\end{proof}

It is worth noting that $z^{k+1}$ is computed before $\gamma_{k+1}$ within the algorithm, so the restriction in \cref{neg_gam_lemma1} and others like it may not be too restrictive to work with.
Our toy numerical tests allowed us to simply put a negative sign in front of the $\gamma_{k}$ terms to attain convergence matching the optimal rate, and which is in some cases markedly faster.
Unfortunately, these results do not give much of an indication as to how exactly the tuning of $\gamma_{k}$ benefits numerical convergence rates.
We leave the theoretical exploration of this phenomena to future work.

\section{Moving anchor in nonconvex/nonconcave minmax problems}

In \cite{lee2021fast}, the methods in \cite{yoon_ryu} are expanded to a broader class of smooth structured nonconvex-nonconcave minimax problems, bringing an $O(1/k^{2})$ rate of convergence rate to a larger class of problems in the setting of minimax games.
This algorithm class is called the FEG, or Fast Extra Gradient method.
We bring the idea of the moving anchor to this more general setting, and show that a moving anchor with more or less the same conditions in the convex-concave setting is also a feasible approach in this class of problems.
Below we give the explicit definition of this FEG modified via a moving anchor, and state its convergence results via a nonincreasing Lyapunov functional and a theorem bounding the squared gradient norm.

The FEG with moving anchor, following \cite{lee2021fast}, is given as

\begin{align}
    \label{feg_moving_alg1} z^{k+1/2} &= z^{k} + \beta_{k}(\bar{z}^{k} - z^{k}) - (1-\beta_{k})(\alpha_{k} + 2\rho_{k})G(z^{k})\\
    \label{feg_moving_alg2} z^{k+1}   &= z^{k} + \beta_{k}(\bar{z}^{k} - z^{k}) - \alpha_{k}G(z^{k+1/2}) - (1-\beta_{k})2\rho_{k}G(z_{k})\\
    \label{feg_moving_alg3} \bar{z}^{k+1} &= \bar{z}^{k} + \gamma_{k+1}G(z^{k+1})\\
    \label{feg_c_def} c_{k+1} &= \frac{c_{k}}{1 + \delta_{k}}\\
    \label{feg_gamma_def} \gamma_{k+1} &= \frac{B_{k+1}}{c_{k+1}(1 + \frac{1}{\delta_{k}})}
\end{align}

\noindent where $\{\delta_{k}\}$ is chosen so that $\displaystyle \sum_{i=0}^{\infty} \log(1 + \delta_{i}) < \infty,$ with $\{\gamma_{k}\}$, and $\{c_{k}\},$ and $c_{\infty}$ chosen in the same method given in the EAG-V with moving anchor, and, as before, $\bar{z}^{0} = z^{0}$.
Before we state the results, two remarks are in order: 
\begin{remark}
	F.
Before we state the results, we point out an additional assumption on the saddle-gradient operator $G$: for some $\rho \in \bigg{(}-\frac{1}{2R},\infty\bigg{)}, \langle G(z) - G(z'), z - z' \rangle \geq \rho\|G(z) - G(z')\|^{2} \; \forall z, z' \in \mathbb{R}^{m}\times \mathbb{R}^{n}$. (Note $z,z'$ are vectors, not matrices.)
	This is known as $\rho-$comonotonicity, and has three sub-conditions.
	For $\rho>0,$ we have cocoercivity; for $\rho=0,$ we have monotonicity; and with $\rho<0$ we have (negative) comonotonicity.
	This condition will hold whenever any FEG variant is discussed throughout this work.
\end{remark}
\begin{remark}
	As in the EAG with moving anchor, one may recover the original fixed anchor FEG by setting $\gamma_{k} = 0$ for all $k.$ 
	This allows us to state our algorithm while also offering an easy reference point for the original fixed anchor version.
\end{remark}

\begin{remark}
    Many of the sequences defined in the following lemma have similar naming conventions to those defined in \cref{LpnvLm_1}. However, the instance of the moving anchor FEG class \eqref{feg_moving_alg1}, \eqref{feg_moving_alg2}, \eqref{feg_moving_alg3} for which we state convergence results utilizes $\alpha_{k} = \frac{1}{R}, \beta_{k} = \frac{1}{k+1}, \rho_{k} = \rho, R_{k} = R$ for $k \geq 0.$ To be clear, \cref{FEG_Lyap_fxnal} is more general and does NOT require these definitions, while \cref{nonconvex_nonconcave_convergence} uses these definitions for explicit convergence results.
\end{remark}

\begin{lemma} \label{FEG_Lyap_fxnal}
    Suppose that the sequences $\{c_{k}\}_{k\geq0}, \{\gamma_{k}\}_{k\geq0}$, are defined as in \eqref{feg_c_def}, \eqref{feg_gamma_def}, and the sequences $\{\alpha_{k}\}_{k\geq0},$ $\{\beta_{k}\}_{k\geq0},$ and $\{R_{k}\}_{k\geq0} \subset (0,\infty),$ and $\{\rho_{k}\}_{k\geq0} \subset \mathbb{R}$ satisfy $\alpha_{0} \in (0,\infty), \alpha_{k} \in (0, \frac{1}{R_{k}}), \beta_{0}=1, \{\beta_{k}\}_{k\geq1} \subseteq (0,1)$ for all $k$.
    Additionally, assume that the following bound, Lipschitz conditions, and comonotonicity conditions respectively hold for a sequence $\{\rho_{k}\} \subset \mathbb{R}$ for all $k\geq0:$
    \begin{align*}
        \frac{(1-\beta_{k+1})}{2\beta_{k+1}}(\alpha_{k+1}+2\rho_{k+1})&-\rho_{k}\leq \frac{1}{2\beta_{k}}(\alpha_{k}+2\rho_{k})-\rho_{k}\\
        \|G(z^{1}) - G(z^{0})\|&\leq R_{0}\|z^{1}-z^{0}\| \\
        \|G(z^{k+1}) - G(z^{k+1/2})\|&\leq R_{k}\|z^{k+1}-z^{k+1/2}\|\\
        \langle G(z^{k+1}) - G(z^{k}),z^{k+1} - z^{k} \rangle& \geq \rho_{k}\|G(z^{k+1}) - G(z^{k})\|^{2}.
    \end{align*}
    If also $\displaystyle A_{0} = \frac{\alpha_{0}(L_{0}^{2}\alpha_{0}^{2}-1)}{2}, B_{0} = 0, B_{1} = 1,$ and
    \[
        A_{k} = \frac{B_{k}(1-\beta_{k})}{2\beta_{k}}(\alpha_{k} + 2\rho_{k}) - B_{k}\rho_{k},\; B_{k+1} = \frac{B_{k}}{1-\beta_{k}},
    \]
    then the Lyapunov functional
    \begin{equation}
        \label{FEG_L_FXNAL} V_{k}:=A_{k}\|G(z^{k})\|^{2}-B_{k}\langle G(z^{k}),\bar{z}^{k} - z^{k} \rangle + c_{k}\|z^{*} - \bar{z}^{k}\|^{2},
    \end{equation}
    where $z^{*}$ is a saddle point, is nonincreasing.
\end{lemma}

\begin{proof}
    This proof proceeds similarly to that of the convex-concave, monotone case in the previous section. First, we write out some relations which will be used shortly:
    \begin{align}
        \label{feg_L_fxnl1}& \qquad z^{k+1} - z^{k} =\frac{\beta_{k}}{1-\beta_{k}}(\bar{z}^{k} - z^{k+1}) - \frac{\alpha_{k}}{1-\beta_{k}}G(z^{k+1/2})-2\rho_{k}G(z^{k})\\
        \label{feg_L_fxnl2}& \qquad z^{k+1} - z^{k} = \beta_{k}(\bar{z}^{k} - z^{k}) - \alpha_{k}G(z^{k+1/2}) - 2\rho_{k}(1-\beta_{k})G(z^{k})\\
        \label{feg_L_fxnl3}& \qquad z^{k+1} - z^{k+1/2} = \alpha_{k}((1-\beta_{k})G(z^{k}) - G(z^{k+1/2}))\\
        \label{feg_L_fxnl4}& \qquad \bar{z}^{k} - \bar{z}^{k+1} = -\gamma_{k+1}G(z^{k+1})
    \end{align}
    As in the proof in the convex-concave case of EAG-V with moving anchor, we introduce a term to the difference of two arbitrary consecutive functionals in our sequence:
    \begin{align}
        \notag &V_{k}-V_{k+1}\\ 
        \geq \notag &A_{k}\|G(z^{k})\|^{2} - B_{k}\langle G(z^{k}), \bar{z}^{k} - z^{k} \rangle - A_{k+1}\|G(z^{k+1})\|^{2} + B_{k+1}\langle G(z^{k+1}), \bar{z}^{k+1} - z^{k+1} \rangle \\
        \notag +&c_{k}\|z^{*}-\bar{z}^{k}\|^{2} - c_{k+1}\|z^{*}-\bar{z}^{k+1}\|^{2}\\ 
        \notag -&\frac{B_{k}}{\beta_{k}}\big{(}\langle G(z^{k+1}) - G(z^{k}),z^{k+1} - z^{k} \rangle - \rho_{k}\|G(z^{k+1}) - G(z^{k})\|^{2}\big{)}\\
        \label{two_inner_prods}=& A_{k}\|G(z^{k})\|^{2} - B_{k}\langle G(z^{k}), \bar{z}^{k} - z^{k} \rangle - A_{k+1}\|G(z^{k+1})\|^{2} + B_{k+1}\langle G(z^{k+1}), \bar{z}^{k+1} - z^{k+1} \rangle \\
        \notag +&c_{k}\|z^{*}-\bar{z}^{k}\|^{2} - c_{k+1}\|z^{*}-\bar{z}^{k+1}\|^{2}\\
        \notag -&\frac{B_{k}}{\beta_{k}}\langle G(z^{k+1}),z^{k+1} - z^{k}\rangle + \frac{B_{k}}{\beta_{k}}\langle G(z^{k}),z^{k+1} - z^{k}\rangle + \frac{B_{k}\rho_{k}}{\beta_{k}}\|G(z^{k+1}) - G(z^{k})\|^{2}
    \end{align}
    From here, we first simplify the introduced term further and then substitute \eqref{feg_L_fxnl1} into the inner product which has a $B_{k}$ out front, and then substitute \eqref{feg_L_fxnl2} into the inner product with a $B_{k+1}$ out front; each of these is in line \eqref{two_inner_prods}.
    After some computation, this leads to
    \begin{align}
        \notag &V_{k} - V_{k+1}\\
        \notag \geq&\big{(}A_{k} - \frac{2B_{k}\rho_{k}(1-\beta_{k})}{\beta_{k}}\big{)}\|G(z^{k})\|^{2} - A_{k+1}\|G(z^{k+1})\|^{2} + \frac{\alpha_{k}B_{k}}{\beta_{k}(1-\beta_{k})}\langle G(z^{k+1}),G(z^{k+1/2})\rangle\\
        \notag +&\frac{2\rho_{k}B_{k}}{\beta_{k}}\langle G(z^{k+1}), G(z^{k})\rangle - \frac{\alpha_{k}B_{k}}{\beta_{k}}\langle G(z^{k}), G(z^{k+1/2})\rangle + B_{k+1}\langle G(z^{k}), \bar{z}^{k+1} - \bar{z}^{k}\rangle\\
        \notag +&\frac{B_{k}\rho_{k}}{\beta_{k}}\|G(z^{k+1}) - G(z^{k})\|^{2} + c_{k}\|z^{*}-\bar{z}^{k}\|^{2} - c_{k+1}\|z^{*}-\bar{z}^{k+1}\|^{2}\\
        \label{second_to_last_line}=& \big{(} A_{k} - \frac{B_{k}\rho_{k}(1-2\beta_{k})}{\beta_{k}}\big{)}\|G(z^{k})\|^{2} - \big{(} A_{k} - \frac{B_{k}\rho_{k}}{\beta_{k}}\big{)}\|G(z^{k+1})\|^{2} + \frac{\alpha_{k}B_{k}}{\beta_{k}(1-\beta_{k})}\langle G(z^{k+1}),G(z^{k+1/2})\rangle\\
        \label{focus_last_three_terms}-&\frac{\alpha_{k}B_{k}}{\beta_{k}}\langle G(z^{k}), G(z^{k+1/2})\rangle + B_{k+1}\langle G(z^{k+1}), \bar{z}^{k+1} - \bar{z}^{k} \rangle + c_{k}\|z^{*}-\bar{z}^{k}\|^{2} - c_{k+1}\|z^{*}-\bar{z}^{k+1}\|^{2}.
    \end{align}
    Next, let's focus on the last three terms in \eqref{focus_last_three_terms}: $B_{k+1}\langle G(z^{k+1}), \bar{z}^{k+1} - \bar{z}^{k} \rangle + c_{k}\|z^{*}-\bar{z}^{k}\|^{2} - c_{k+1}\|z^{*}-\bar{z}^{k+1}\|^{2}.$
    By Cauchy-Schwartz,
    \[
        \|z^{*}-\bar{z}^{k+1}\|^{2} \leq (1 + \delta_{k})\|z^{*}-\bar{z}^{k}\|^{2} + (1+\frac{1}{\delta_{k}})\|\bar{z}^{k} - \bar{z}^{k+1}\|^{2}.
    \]
    Second, by construction
    \[
        B_{k+1}\langle G(z^{k+1}), \bar{z}^{k} - \bar{z}^{k+1} \rangle = \frac{B_{k+1}}{\gamma_{k+1}}\|\bar{z}^{k} - \bar{z}^{k+1}\|^{2}
    \]
    and
    \[
        c_{k+1}\leq \frac{c_{k}}{1+\delta_{k}}, \; \gamma_{k+1} \leq \frac{B_{k+1}}{c_{k+1}(1 + \frac{1}{\delta_{k}})}.
    \]
    Applying these facts to the three terms we're considering, we get that
    \begin{align*}
        &B_{k+1}\langle G(z^{k+1}), \bar{z}^{k+1} - \bar{z}^{k} \rangle + c_{k}\|z^{*}-\bar{z}^{k}\|^{2} - c_{k+1}\|z^{*}-\bar{z}^{k+1}\|^{2}\\
        \geq& \frac{B_{k+1}}{\gamma_{k+1}}\|\bar{z}^{k+1} - \bar{z}^{k}\|^{2} + c_{k}\|z^{*}-\bar{z}^{k}\|^{2} - c_{k+1}\big{(}(1+\delta_{k})\|z^{*}-\bar{z}^{k}\|^{2} + (1+\frac{1}{\delta_{k}})\|\bar{z}^{k} - \bar{z}^{k+1}\|^{2}\big{)}\\
        \geq& \frac{B_{k+1}}{B_{k+1}}c_{k+1}(1 + \frac{1}{\delta_{k}})\|\bar{z}^{k} - \bar{z}^{k+1}\|^{2} + c_{k}\|z^{*}-\bar{z}^{k}\|^{2} \\
        &- c_{k+1}(1+\delta_{k})\|z^{*}-\bar{z}^{k}\|^{2} - c_{k+1}(1+\frac{1}{\delta_{k}})\|\bar{z}^{k} - \bar{z}^{k+1}\|^{2}\\
        \geq&c_{k}\|z^{*}-\bar{z}^{k}\|^{2} - c_{k+1}(1+\delta_{k})\|z^{*}-\bar{z}^{k}\|^{2} \geq c_{k}\|z^{*}-\bar{z}^{k}\|^{2} - c_{k}\|z^{*}-\bar{z}^{k}\|^{2} \geq 0.
    \end{align*}
    
    While this takes care of the latter three terms in lines \eqref{second_to_last_line} to \eqref{focus_last_three_terms}, that everything else is nonnegative is a nontrivial argument.
    However, it directly follows the proof of Lemma 7.1 in \cite{lee2021fast}, so as before we refer to their proof, and then our Lyapunov functional is also nonincreasing.
\end{proof}

\begin{theorem}[$O(1/k^{2})$ convergence rate for FEG with moving anchor] \label{nonconvex_nonconcave_convergence}
    For the $R-$Lipschitz continuous and $\rho-$comonotone operator G where $\rho>-\frac{1}{2R}$, and $z^{*} \in Z_{*}(G), Z_{*}(G) := \{z^{*} \in \mathbb{R}^{d}: G(z^{*}) = 0\},$ and $c_{\infty} - \frac{1}{\frac{1}{R} + 2\rho} \geq 0,$ the sequence $\{z^{k}\}_{k\geq0}$ generated by FEG with moving anchor satisfies
    \[
        \|G(z^{k})\|^{2} \leq \frac{4c_{0}\|z^{0}-z^{*}\|^{2}}{k^{2}(\frac{1}{R} + 2\rho)}
    \]
    for all $k\geq1.$
\end{theorem}

\begin{proof}
    Under the same assumptions as \cref{FEG_Lyap_fxnal}, we take $\alpha_{k} = 1/R, \beta_{k} = \frac{1}{k+1}, R_{k}=R, \rho_{k} = \rho,$ which satisfy the conditions in the statement for all $k$ greater than or equal to $0.$
    These give us $B_{k} = k, A_{k} = \frac{k^{2}}{2}(\frac{1}{R} + 2\rho) - k\rho.$
    
    From here,
    \[
        c_{0}\|z^{*} - z^{0}\|^{2} = V_{0} \geq V_{k} = \bigg{(}\frac{k^{2}}{2}(\frac{1}{R} + 2\rho) -k\rho\bigg{)}\|G(z^{k})\|^{2} - k\langle G(z^{k}), \bar{z}^{k} - z^{k} \rangle + c_{k}\|z^{*} - \bar{z}^{k}\|^{2},
    \]
    so then
    \begin{align*}
        \frac{k^{2}}{2}&(\frac{1}{L} + 2\rho)\|G(z^{k})\|^{2} + c_{k}\|z^{*} - \bar{z}^{k}\|^{2}\\
        \leq& k\langle G(z^{k}), \bar{z}^{k} - z^{k} \rangle + k\rho\|G(z^{k})\|^{2} + c_{0}\|z^{*} - z_{0}\|^{2}\\
        \leq&k\langle G(z^{k}), \bar{z}^{k} - z^{*} \rangle + c_{0}\|z^{*} - z_{0}\|^{2} \text{ (by comonotonicity condition)}\\
        \leq&k\|G(z^{k})\| \|\bar{z}^{k} - z^{*}\| + c_{0}\|z^{*} - z_{0}\|^{2}\\
        \leq&\frac{k^{2}}{2\delta}\|G(z^{k})\|^{2} + \frac{\delta}{2}\|\bar{z}^{k} - z^{*}\|^{2} + c_{0}\|z^{*} - z_{0}\|^{2}.
    \end{align*}
    From here, define $\frac{1}{\delta} = \frac{1}{2R} + \rho$.
    Then we have that 
    \[
        \frac{k^{2}}{2}\bigg{(} \frac{1}{R} + 2\rho - \frac{1}{2R} - \rho\bigg{)}\|G(z^{k})\|^{2} + \bigg{(} c_{\infty} - \frac{1}{\frac{1}{R} + 2\rho}\bigg{)}\|\bar{z}^{k} - z^{*}\|^{2} \leq c_{0}\|z^{*} - z_{0}\|^{2}, 
    \]
    and as long as the constant $c_{\infty} - \frac{1}{\frac{1}{R} + 2\rho} \geq 0,$ we obtain the desired result by dividing both sides of the inequality
    \[
        \frac{k^{2}}{2}\bigg{(} \frac{1}{2R} + \rho\bigg{)}\|G(z^{k})\|^{2} \leq c_{0}\|z^{*} - z_{0}\|^{2}
    \]
    by $\frac{k^{2}}{2}\bigg{(} \frac{1}{2R} + \rho\bigg{)}$.
\end{proof}
See \cite{lee2021fast}'s proof of Theorem 4.1 for the analogous result with a fixed anchor.
Next, we show that having $-\gamma_{k+1}$ in place of $\gamma_{k+1}$ may also, with some additional assumptions, provide a convergent algorithm.

\begin{lemma} \label{neg_gam_lemma2}
    In the setting of \cref{FEG_Lyap_fxnal}, replace $\gamma_{k}$ with $-\gamma_{k}$ in the definition of the FEG algorithm with moving anchor, and suppose $\displaystyle \gamma_{k+1} = \min \frac{B_{k+1}}{c_{k+1}(1 + \frac{1}{\delta_{k}})}, \frac{e_{k+1}}{2B_{k+1}\|G(z^{k+1})\|^{2}},$ where $\sum e_{k} < \infty.$ Then the Lyapunov functional described in \cref{FEG_Lyap_fxnal} is nonincreasing, and we attain the same order of convergence for the FEG with moving anchor and $-\gamma_{k}.$
\end{lemma}

\begin{proof}
    The proof proceeds in exactly the same manner as that in \cref{neg_gam_lemma1}.
\end{proof}

As in the EAG-V with moving anchor case, we suspect this restriction is not too major a restriction based off of numerical results, and that there is a `better' way to show that the $-\gamma_{k}$ version of our algorithm converges.

\section{Numerical experiments}

In this section we detail several numerical experiments. First, we visualize two thousand iterations of EAG-V and FEG, each moving anchor versus the fixed anchor, on a toy `almost bilinear' example. Next, we look at the log of the grad norm squared versus the log of iterations for the EAG examples. Note that this error graph is an example in the monotone convex-concave case. We then run a nonconvex-nonconcave negative comonotone example for FEG variants, where some interesting convergence behaviors among the moving anchor variants are exhibited. Finally, we study monotone FEG variants (moving and fixed anchor) on a nonlinear two player game. Throughout all of these examples, $c_{1} = \pi^{2}/6,$, $c_{k} = \frac{c_{k-1}}{1+\delta_{k-1}} (k=2,3,...),$ and in all except for the last example, $\delta_{k}$ is chosen to be $\exp(1/k^{2}) - 1.$

\begin{figure}[h!]
    \centering
    \includegraphics[scale=.5]{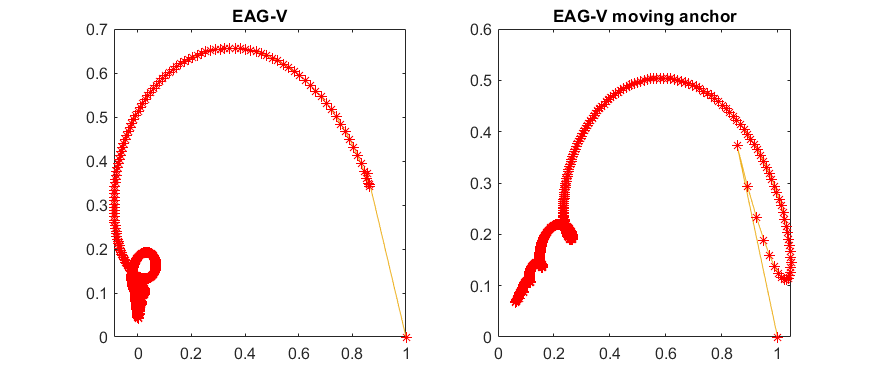}
    \caption{The first two thousand iterations of the EAG algorithm with varying step-size, or EAG-V, compared to the first two thousand iterations of the moving anchor EAG-V algorithm.}
    \label{fig:fixed_vs_moving_EAGV}
\end{figure}

\cref{fig:fixed_vs_moving_EAGV} compares the iterations of EAG-V with a fixed anchor to the iterations for the moving anchor EAG-V. 
\cref{fig:fixed_vs_moving_EAGV}, \cref{fig:moving_EAGV_anchors}, and \cref{fig:moving_FEG_anchors} all display iterations where the function used is the `almost bilinear' function $f:\mathbb{R}^{2} \to \mathbb{R}, f(x,y) = \epsilon \frac{\|x\|^2}{2} + \langle x,y \rangle - \epsilon \frac{\|y\|^2}{2}.$
Here, $\epsilon$ is small, for these experiments set to $0.01,$ and the straightforward nature of the example allows for ease of visualizing the iterations as well as their differences when it comes to comparing convergence rates. 
In particular, the unique saddle-point is $(0,0).$

\begin{figure}[h!]
	\centering
	\includegraphics[scale=.6]{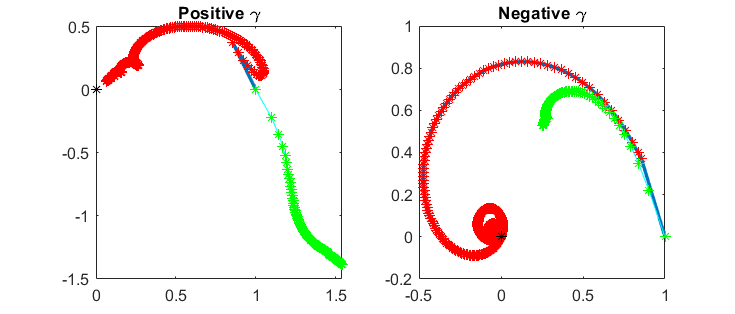}
	\caption{The two moving anchor EAG-V variants compared in red, along with their anchors in green.}
	\label{fig:moving_EAGV_anchors}
\end{figure}

\cref{fig:moving_EAGV_anchors} compares, via the same function as \cref{fig:fixed_vs_moving_EAGV}, the two moving anchor variants of EAG-V.
When the $\gamma_{k}$ parameter is positive, the anchor iterations moves away from the saddle and the algorithm updates very rapidly.
When $\gamma_{k}$ has only its sign changed to negative, the anchor (seen in green) seems to stay much closer to the iterations and the saddle-point.
The iterations appear to converge at a markedly faster rate (by a constant) for this latter case over both the fixed anchor and the positive $\gamma_{k}$ setting, an observation that is confirmed below.

\begin{figure}[h!]
	\centering
	\includegraphics[scale=.6]{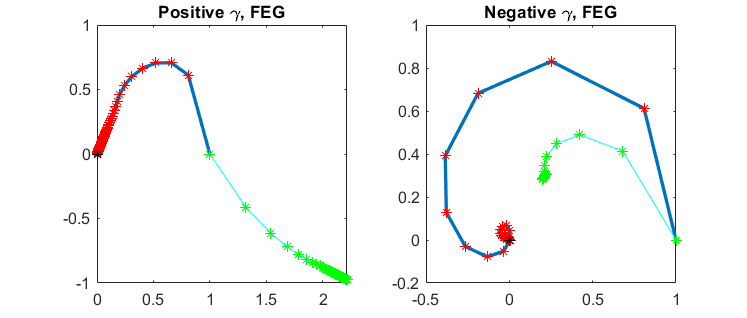}
	\caption{The two moving anchor FEG variants compared in red, along with their anchors in green.}
	\label{fig:moving_FEG_anchors}
\end{figure}

\cref{fig:moving_FEG_anchors} compares the two moving anchor version of the FEG method, in the same manner as the comparison shown in \cref{fig:moving_EAGV_anchors}: red dots are the algorithm updates, green dots are the anchor updates, and the function is the `almost bilinear' one previously described.
Note that in both cases, the iterations seem to zone in on and converge to the saddle point in a much faster manner.
In \cite{lee2021fast}, the authors established that even on convex-concave problems, FEG performs at the same optimal order of convergence as EAG, but at a significantly faster rate.
This behavior seems to have carried over to our algorithm where we introduce the moving anchor to these frameworks.

\begin{figure}[h!]
	\centering
	\includegraphics[scale=.5]{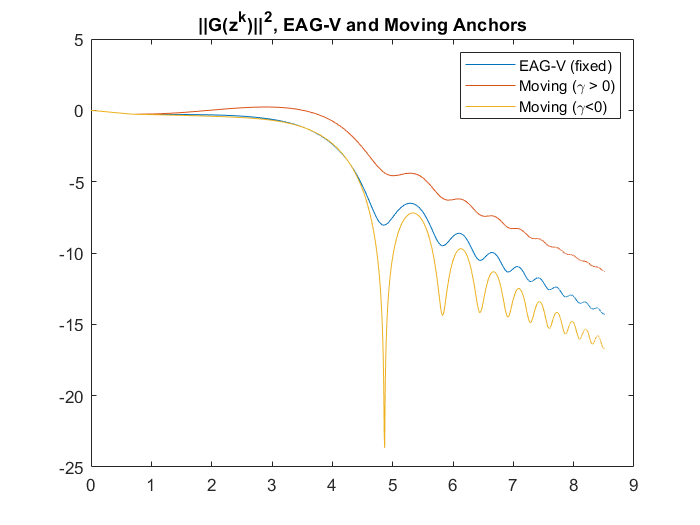}
	\caption{Comparison of the grad-norm squared of three EAG-V variants of interest on a toy `almost bilinear' problem.}
	\label{fig:grad_norms_sqrd_eagv}
\end{figure}

\cref{fig:grad_norms_sqrd_eagv} captures the behavior of $\|G(z^{k})\|^{2}$ across all three convex-concave algorithms of interest: EAG-V, moving anchor EAG-V with positive $\gamma_{k},$ and moving anchor EAG-V with negative $\gamma_{k}.$
Each algorithm attains the optimal order of convergence, while the negative $\gamma_{k}$ algorithm is markedly faster than both algorithms by a constant.
Identical behavior occurrsed under the same problem setting with the FEG and FEG with moving anchors (positive and negative $\gamma_{k}$), with the negative $\gamma_{k}$ algorithm again being the fastest, so we do not include this figure here.

\begin{figure}[h!]
	\centering
	\includegraphics[scale=.5]{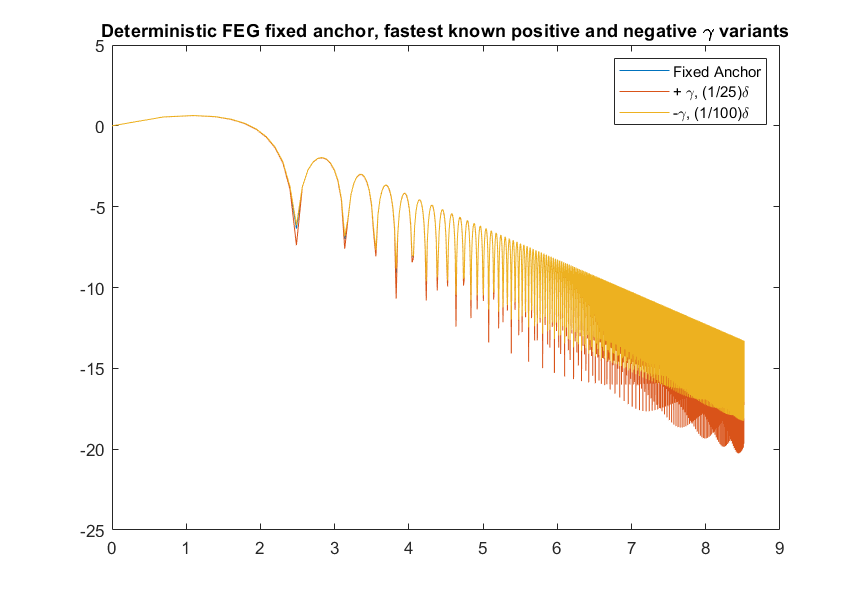}
	\caption{Comparison of the errors of three FEG variants in a nonconvex-nonconcave setting. Note the positive $\gamma$ with $\delta$ scaled by $1/25$ converges fastest.}
	\label{fig:grad_norms_sqrd_worse_FEG}
\end{figure}

\cref{fig:grad_norms_sqrd_worse_FEG} captures the error of FEG across all three anchor variants in a numerical example that is explicitly negative comonotone with a nonconvex-nonconcave objective:

\[L(x,y) = \frac{\rho R^{2}}{2}x^{2} + R\sqrt{1 - \rho^{2}R^{2}}xy - \frac{\rho R^{2}}{2}y^{2}\]

with $L: \mathbb{R}^{2} \to \mathbb{R}, R = 1, \rho = -1/3$ $1-$smooth and $-1/3$-negative comonotone.
Interestingly, in this scenario a variant of the positive $\gamma_{k}$ moving anchor algorithm is the fastest when $\delta_{k}$ is scaled by $1/25$, in sharp contrast to the result displayed in \cref{fig:grad_norms_sqrd_eagv}.
The fixed anchor and the negative $\gamma_{k}$ values seem to almost coincide.
This result suggests that different problem settings offer different optimal anchoring choices, and that in the negative comonotone setting, the fixed anchor FEG may be nearly optimal.



Finally, \cref{fig:nonlinear_games}, compares three different monotone FEG variants on a particular nonlinear game that was studied extensively in \cite{chen2014optimal}:
\[\min_{x \in \Delta^{n}} \max_{y \in \Delta^{m}} \frac{1}{2}\langle Qx,x\rangle + \langle Kx,y \rangle \]
where $Q = A^{T}A$ is positive semidefinite for $A \in \mathbb{R}^{k \times n}$ which has entries generated independently from the standard normal distribution, $K \in \mathbb{R}^{m \times n}$ with entries generated uniformly and independently from the interval $[-1,1],$ and $\Delta^{n}, \Delta^{m}$ are the $n-$ and $m-$simplices, respectively:
\[\Delta^{n} := \Big{\{} x\in\mathbb{R}_{+}^{n}: \sum_{i=1}^{n}x_{i} = 1 \Big{\}}, \; \Delta^{m} := \Big{\{} y\in\mathbb{R}_{+}^{m}: \sum_{j=1}^{m}y_{j} = 1 \Big{\}}.\]
One may interpret this as a two person game where player one has $n$ strategies to choose from, choosing strategy $i$ with probability $x_{i} \; (i = 1, . . . , n)$ to attempt to minimize a loss, while the second player attempt to maximize their gain among $m$ strategies with strategy $j$ chosen with probability $y_{j} \; (j = 1, \ldots , m).$
The payoff is a quadratic function that depends on the strategy of both players.
This was implemented following the 3 operator splitting scheme in \cite{davis2015three}, with parameter $\lambda = 0.25$.

For this example, we used FEG fixed and moving anchor variants in the monotone (that is, $\rho=0$) setting of the algorithm.
For the high dimensional setting, we compute $20,000$ iterations and view the log of the grad norm squared of the fixed anchor versus the positive and negative $\gamma_{k}$ moving anchor variants.
For the low dimensional setting, we compare the same algorithms but only compare this on $8,000$ iterations, as the convergence behavior in this lower dimensional setting is more quickly distinguished.

In both settings, the positive $\gamma_{k}$ variant is the fastest (ie, most accelerated) algorithm by a significant margin, even with different random seeds selected across both test runs.
This is a significant contrast to the EAG-V toy example, also a convex-concave problem, where the $-\gamma_{k}$ variant of the moving anchor is the fastest algorithm.
Taken together, both results are promising for the moving anchor framework, but suggest that more theoretical work is to be done to understand the acceleration mechanism offered by anchoring variants, and what variant will be fastest in a given problem.

\begin{figure}[h]

\begin{subfigure}{0.5\textwidth}
\includegraphics[width=0.9\linewidth, height=6cm]{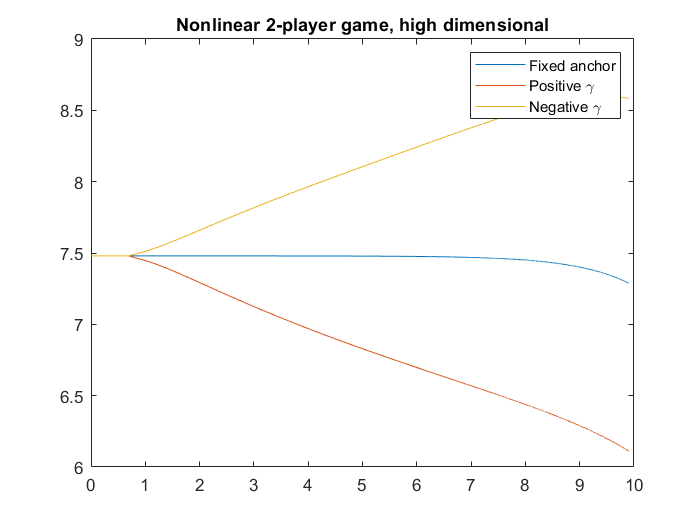} 
\caption{High dimensional nonlinear game.}
\label{fig:high_dim_game}
\end{subfigure}
\begin{subfigure}{0.5\textwidth}
\includegraphics[width=0.9\linewidth, height=6cm]{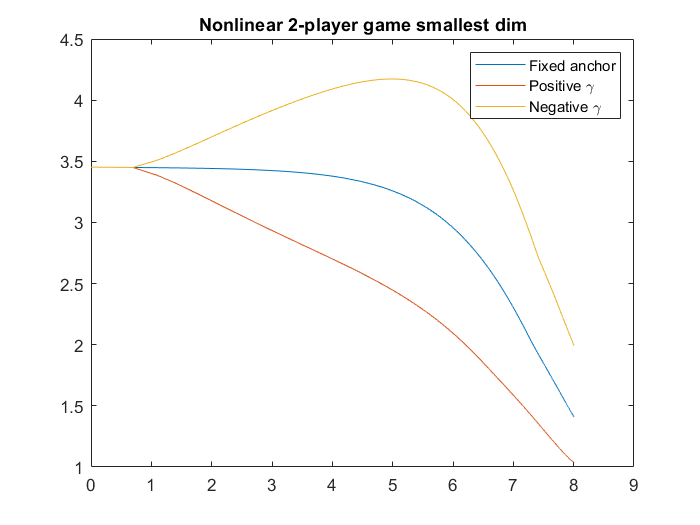}
\caption{Low dimensional nonlinear game.}
\label{fig:low_dim_game}
\end{subfigure}

\caption{High dimensional and low dimensional variants of 2-player nonlinear game with anchoring algorithms}
\label{fig:nonlinear_games}
\end{figure}

\clearpage

\section{Conclusion}

The moving anchor variants of anchored acceleration methods retain optimal convergence rates and also demonstrate superior-to-comparable numerical performance with some parameter tuning.
The optimal order of convergence is obtained across different problem settings, from convex-concave to negative co-monotone problems. 
Interestingly, across numerous problem settings there exists a version of the moving anchor algorithm, parametrized by $\gamma_{k},$ that demonstrates superior numerical performance compared to other state-of-the-art algorithms.
The variety of numerical examples demonstrates a wide array of applications for our algorithms in both theoretical and applied settings.
Of future interest, one may consider parallelized/asynchronous implementations of moving anchor saddle point algorithms, a tighter analysis of $-\gamma_{k}$ convergence, a theoretical understanding of how the $\gamma_{k}$ and other parameters such as $\delta$ affect convergence rates, implementations of the moving anchor in other Halpern-adjacent frameworks, and the identification of problem settings which our moving anchor may exploit effectively among many other topics.

\section*{Acknowledgments}

We thank Donghwan Kim, Ernest K. Ryu, and Taeho Yoon for their invaluable suggestions that have greatly improved and inspired this work.

\clearpage

\bibliographystyle{amsplain}
\bibliography{bibliography}

\end{document}